\documentclass[12pt, 14paper,reqno]{amsart}
\usepackage{amsmath,amsfonts,amssymb}
\usepackage{mathrsfs}
\usepackage{longtable}
\usepackage[breaklinks]{hyperref}
\usepackage{graphicx}

\newenvironment{dedication}
{\vspace{0ex}\begin{quotation}\begin{center}\begin{em}}
        {\par\end{em}\end{center}\end{quotation}}
\theoremstyle{plain}
\newtheorem{thm}{Theorem}[section]
\newtheorem{lem}{Lemma}[section]
\newtheorem{cor}{Corollary}[section]

\newtheorem{thma}{Theorem}

\theoremstyle{proof}

\numberwithin{equation}{section}

\begin{document} 
\title[A class of Lebesgue-Ramanujan-Nagell equations]{On a class of Lebesgue-Ramanujan-Nagell equations}
\author{Azizul Hoque}
\address{Department of Mathematics, Faculty of Science, Rangapara College, Rangapara, Sonitpur-784505, Assam, India.}
\email{ahoque.ms@gmail.com}
\date{\today}
\maketitle

\begin{dedication}
To Professor Michel Waldschmidt on his $75^{th}$ birthday with great respect
\end{dedication}

\begin{abstract}
We deeply investigate the Diophantine equation $cx^2+d^{2m+1}=2y^n$ in integers $x, y\geq 1, m\geq 0$ and $n\geq 3$, where $c$ and $d$ are coprime positive integers satisfying $cd\not\equiv 3 \pmod 4$. We first solve this equation for prime $n$, under the condition $\gcd(n, h(-cd))=1$, where $h(-cd)$ denotes the class number of the imaginary quadratic field $\mathbb{Q}(\sqrt{-cd})$. We then completely solve this equation for both $c$ and $d$ primes, under the assumption $\gcd(n, h(-cd))=1$. We also completely solve this equation for $c=1$ and $d\equiv 1 \pmod 4$, under the condition $\gcd(n, h(-d))=1$. For some fixed values of $c$ and $d$,  we derive some results concerning the solvability of this equation.  
\end{abstract}

\section{Introduction}
Many special cases of the Lebesgue-Ramanujan-Nagell type equation 
\begin{equation}\label{eqlrn}
x^2+d=2y^n, ~~  x, y\geq 1, \gcd(x,y)=1, n\geq 3,
\end{equation}
where $d\geq 1$ is an integer,  have been considered over the years. In 1895, St\"ormer \cite{ST95} proved using an elementary factorization argument  that for odd $n$ and $d=1$, \eqref{eqlrn} has no solution with $y>1$. Later, Cohn extended this result in \cite{CO96}, and  proved that $(x,y,n)=(239,13,4)$ is the only solution of \eqref{eqlrn} with $y>1$ when $d=1$. In \cite{PT00},  Pink and Tengely considered \eqref{eqlrn} for $d=a^2$ and gave an upper bound for the exponent $n$ depending only on $a$. Moreover, they completely solved it when $1 < a < 1000$ and $3< n< 80$. For $d=a^2$ with odd $a$ and $3\leq a\leq 501$, Tengely \cite{TE04} extended the work of \cite{PT00} and completely solved \eqref{eqlrn} when $n$ is prime. On the other hand, for $d=q^{2m}$ with $m\geq 1$ and prime $q\geq 3$, Tengely \cite{TE07} established some bounds for prime $n$. Applying these bounds, he completely solved \eqref{eqlrn} for $d=q^{2m}$ and $ y=17$ (resp. $d=3^{2m}$) with $m\geq 1$ and prime $q\geq 3$. Later in \cite{ZLT12}, Zhu et al. extended the work of Tengely \cite{TE07}, and described all possible solutions of \eqref{eqlrn} for $d=q^{2m}$ with $q$ an odd prime and $m\geq 1$. Muriefah et al. investigated \eqref{eqlrn} in \cite{MLST09}, and they completely solved it when $d\equiv 1\pmod 4$ and less than $100$. On the other hand, Ljunggren \cite{LJ66} considered a general Lebesgue-Ramanujan-Nagell equation,
\begin{equation}\label{eql}
cx^2+d=2y^n,~~ x, y\geq 1, n\geq 3, \gcd(cx, y)=1, \gcd(h(-cd), n)=1,
\end{equation}
where $c$ and $d$ are fixed positive integers, and $h(\Delta)$ denotes the class number of the quadratic field $\mathbb{Q}(\sqrt{\Delta})$. He described the solutions of some very particular cases of \eqref{eql}, under certain conditions on  $c, d$ and $n$. In \cite{GM20},  Ghanmi and Abu Muriefah extended this work of Ljunggren by utilizing the classical results on the primitive divisors of Lehmer sequence due to Bilu, Hanrot and Voutier (see \cite{BH01} and the references therein). Precisely, they investigated the solvability of \eqref{eql} in positive integers $x, y$ and odd prime number $n$, when $cd\not\equiv 3\pmod 4$ and $cd$ is square-free.  

The purpose of this paper is to perform a deep investigation on the Labesgue-Ramanujan-Nagell type equation 
\begin{equation}\label{eqg}
cx^2+d^{2m+1}=2y^n,~~ x, y\geq 1, m\geq 0, n\geq 3,
\end{equation}
where $c$ and $d$ are given positive integers with square-free $cd$. This work generalizes some of the earlier works; for instance, the work of Ljunggren \cite{LJ66}, and that of Ghanmi and Abu Muriefah \cite{GM20}. In particular, we prove three theorems concerning the solvability of \eqref{eqg}, and then we derive some consequences for certain fixed values of $c$ and $d$.  The first theorem generalizes the work of Ghanmi and Abu Muriefah \cite{GM20}. Note that the elegant results of Bilu et al. \cite{BH01} and Voutier \cite{VO95} on the existence of primitive divisors of Lehmer numbers have turned out to be extremely powerful tools in the proofs of our results too. The  factorization argument used in  \cite{GM20}  is based on the work of Ljunggren \cite{LJ66}, which was later  presented in a better way by Yuan in \cite{YU05}. We utilize this work of Yuan.     
We first solve \eqref{eqg} for odd prime $n$ under the conditions $\gcd(h(-cd), n)=1$ and $cd\not\equiv 3\pmod 4$. Precisely, we prove:

\begin{thm}\label{thm}
Let $c$ and $d$ be positive integers such that $cd$ is square-free and $cd\not\equiv 3\pmod 4$. Assume that $p$ is an odd prime such that $p\nmid h(-cd)$. 
\begin{itemize}
\item[(I)]
The Diophantine equation
\begin{equation}\label{eqp}
cx^2+d^{2m+1}=2y^p
\end{equation} 
has no positive integer solution $(x,y,m,p)$  for $p>7$, except for $(x,y)=(1,1)$.
\item[(II)] If $p=7$, then \eqref{eqp} has only one solution, viz. $(x,y,m)=(1169,9, 0)$ with $(c,d)=(7, 11)$, provided $(x,y)\ne(1,1)$.
\item[(III)] If $p=5$ and $cd\ne 1$, then for any $m\geq 0$ the solutions of \eqref{eqp} are  given by 
\begin{align*}
(c, d, x,y)\in&\Big\{ \left(F_{k-2\varepsilon}/2u^2, L_{k+\varepsilon}/2v^2, u\left((2F_k-F_{k-2\varepsilon})^2+F_kL_{k+\varepsilon}\right), F_k\right), \nonumber\\ &
\left(L_{t-2\varepsilon}/2u^2, 5F_{t+\varepsilon}/2v^2, u\left( (2L_t-L_{t-2\varepsilon})^2+5L_tF_{t+\varepsilon}\right), L_t\right)\Big\},
\end{align*}
where $u$ and $v$ are suitable coprime odd positive integers, $k\geq 3$ and $t\geq 2$ are integers, $F_k$ (resp. $L_k$) is $k$-th Fibonacci (resp. Lucas) number and $\varepsilon=\pm 1$.  
\item[(IV)] If $p=3$ and $cd\ne 1$, then the solutions of \eqref{eqp} are given by
$$
(c, d, x, y, m)=\left(c, \frac{3u^2c-2}{v^2}, u(4u^2c-3), 2u^2c-1, m\right), 
$$
where $v=d^m$, $u\geq 1$  is an odd integers and $k\geq 2$ is an integer. 
\end{itemize}
\end{thm}
\subsection*{Remarks} We make the following comments on Theorem \ref{thm}. 
\begin{itemize}
\item[(i)] There are many cases where both $F_{k+2}$ and $F_{k-2}$ are odd. In fact, both $F_{k+2}$ and $F_{k-2}$ are  odd when $3\mid k$. In such cases, $c=F_{k-2\varepsilon}/2u^2$ can not be an integer. Therefore the solutions mentioned in (III) may not be existed.  
\item[(ii)] The work of Ghanmi and Abu Muriefah \cite{GM20} follows from Theorem \ref{thm}. In particular, we get \cite[Theorem 1]{GM20} from Theorem \ref{thm} by putting $m=0$. Note that $(q, D, C) = (5, 11, 7)$ should be read as $(q, D, C) = (7, 11, 7)$ in \cite[Theorem 1]{GM20}. 
\end{itemize}
Theorem \ref{thm} yields the following straightforward corollary.
\begin{cor}\label{cor1}
Let $c\equiv 1\pmod 4$ be a square-free positive integer and $p $ an odd prime such that $p\nmid h(-c)$. Then
the Diophantine equation
$cx^2+1=2y^p $
has no positive integer solution $(x,y,p)$, except for $(x, y)=(1,1)$.
\end{cor}

The second result is about the integer solutions of \eqref{eqg}, when both $c$ and $d$ are distinct odd primes satisfying $cd\not\equiv 3\pmod 4$ and $n\geq 3$ is any integer other than powers of $2$. Particularly, we prove:

\begin{thm}\label{thm2}
Let $\ell$ and $q$ be two distinct odd primes such that $\ell q\not\equiv 3\pmod 4$. Suppose that $n>2$ is an integer which is not a power of $2$ and $\gcd(n, h(-\ell q))=1$. Then the following statements hold.
\begin{itemize}
\item[(i)]
The Diophantine equation
\begin{equation}\label{eqn}
\ell x^2+q^{2m+1}=2y^n
\end{equation} 
has no positive integer solution $(x,y,m,n)$  when $n$ has a prime divisor bigger than $7$.
\item[(ii)] If $7$ is the largest prime divisor of $n$, then the solutions of \eqref{eqn} exist only when $(\ell,q)=(7,11)$, which are:  $$(x,y,m, n)\in\{(1169,9, 0, 7), (1169,3, 0, 14)\}.$$
\item[(iii)] If $5$ is the largest prime divisor of $n$, then the solutions \eqref{eqn} are given by {\small{
$$\hspace{13mm}(x,y,n)=\left(\frac{1}{4}\left|\sqrt{\frac{F_{k-2\varepsilon}}{2\ell}} \left(F_{k-2\varepsilon}^2-\frac{5}{2}F_{k-2\varepsilon}L_{k+\varepsilon}+\frac{5}{4}L_{k+\varepsilon}^2\right)\right|,  \left(\frac{F_{k-2\varepsilon}+L_{k+\varepsilon}}{4}\right)^{\frac{1}{2^{r}3^{s}5^{t}}}, 2^{r}3^{s}5^{t+1}\right),$$}}
\hspace{-1.5mm}where $F_k$ and $L_k$ are respectively Fibonacci and Lucas numbers satisfying the negative Pell equation \eqref{5xxx} and $\ell, q, m$ are given by \eqref{5x}. Here $r, s$ and $t$ are suitable non-negative integers.
\item[(iv)] If $3$ is the largest prime divisor of $n$, then the solutions of \eqref{eqp} are given by
$$\hspace{13mm}(\ell, q,x,y,m, n)=\left(\frac{q^{2m+1}+2}{3u^2}, q, \frac{u(4q^{2m+1}-1)}{3},  \left(\frac{2q^{2m+1}+1}{3}\right)^{1/2^r3^t}, m, 2^r3^{t+1}\right),$$
where $u\geq 1$ is a suitable odd integer, and $r\geq 0$  and $t\geq 0$ are any suitable integers.   
\end{itemize}
\end{thm}
We now consider $\ell \in\{3,7,11,19\}$ and $q\in \{43,79,87,119,239\}$. Then $\ell q\not\equiv 3 \pmod 4$ and we see that $2$ and $3$ are the only prime factors of $h(-\ell q)$. Thus as a consequence of Theorem \ref{thm2}, one gets the following straightforward corollary. 
\begin{cor}\label{cor1}
Fix $\ell \in\{3,7,11,19\}$ and $q\in\{43,79,87,119,239\}$. If $n$ is a positive integer such that $\gcd(n, 30)=1$, then \eqref{eqn} has no positive integer solution. 
\end{cor}
If we fix $\ell, q \in\{3,7,11,23\}$,  then $\ell q\not\equiv 3 \pmod 4$ and $2$ is the only prime factor of $h(-\ell d)$. Therefore by Theorem \ref{thm2}, we get the following:
\begin{cor}\label{cor2}
Let $\ell, q \in\{3,7,11,23\}$ such that $\ell \ne q$. Then for an odd integer $n>1$, the positive integer solutions of  \eqref{eqn} are given by $$(\ell, q, x,y,m,n)\in\left\{ (7,11,1169,9,0,7), \left(3, 7, 9, 5, 0, 3\right)\right\}.$$ 
\end{cor}

In the next result, we fix $c=1$ and completely solve \eqref{eqg} under the restriction $\gcd(n, h(-d))=1$. This result gives a particular case of Theorem \ref{thm} when $n$ is a prime.

\begin{thm}\label{thm3}
Let $d\equiv 1\pmod 4$ be a square-free integer with $d>1$. Assume that $n>2$ is an integer which is not a power of $2$ such that  $\gcd(n, h(-d))=1$. Then the Diophantine equation
\begin{equation}\label{eqd}
x^2+d^{2m+1}=2y^n,
\end{equation} 
has no positive integer solution when $n$ has a prime divisor other than $2$ and $3$. When $2$ and $3$ are the only prime divisors of $n$, all positive integer solutions are given by $$(d,x,y, m, n)=\left((3u^2-2)/v^2, 4u^3-3u, (2u^2-1)^{1/2^r3^t},  m, 2^r3^{t+1}\right),$$ where $v=d^m$, $u\geq 1$ is an odd integer coprime to $v$, and $r$ and $t$ are suitable non-negative integers.
\end{thm}
\subsection*{Remarks} We make the following remarks about Theorem \ref{thm3}.
\begin{itemize}
\item[(i)] For $d\equiv2 \pmod 4$, \eqref{eqd} has no solution.
\item[(ii)] For $m=0$, \eqref{eqd} reduces to $x^2+d=2y^n$, which was deeply investigated by Abu Muriefah et al. \cite{MLST09} for $1\leq d<100$, under the condition $\gcd(x,y)=1$.  However, the following solutions remained unnoticed to them 
$$(d,x,y,n)\in\{(73, 485,49,3), (73, 485,7,6)\},$$ 
which are obtained by our Theorem \ref{thm3}.
\item[(iii)] Since in \cite{MLST09} the authors were allowed $d$ to be square, they failed to determine $(d, x,y,n)=(25, 99,17,3)$. Although it is out of the assumptions of Theorem \ref{thm3}, we obtain this solution from our theorem.
\item[(iv)] For $(m,n)=(0,3)$, the positive integer solutions of \eqref{eqd} are given by $$(d,x,y)=(3u^2-2,4u^3-3u, 2u^2-1),$$
where $u\geq 1$ is an odd integer. It is noted that each fixed $d$ gives at most one positive integer solution $(x,y)$.    
\item[(v)] From Theorem \ref{thm3}, we see that $x^2+d=2y^6$ has an infinite family of positive integer solutions, corresponding to the solutions of the negative Pell equation $y^2-2u^2=-1$. More precisely, these solutions are given by 
$$(d,x,y)=(3u_t^2-2,  4u_t^3-3u_t, y_t),$$
where $(u_t, y_t)=(2y_{t-1}+3u_{t-1},  3y_{t-1}+4u_{t-1})$ with $t\geq 0$ an integer and $(u_0,y_0)=(1,1)$. 
\end{itemize}

\section{Lehmer sequences and some lemmas}
We need some definitions and notations to state some crucial results on the existence of primitive divisors of Lehmer numbers. Suppose that $\alpha$ and $\beta$ are algebraic integers. Then $(\alpha, \beta)$ is a {\it Lehmer pair} if $(\alpha + \beta)^2$ and $\alpha\beta$ are two non-zero coprime rational integers, and $\alpha/\beta$ is not a root of unity. For a positive integer $\ell$ and a Lehmer pair $(\alpha, \beta)$, the corresponding sequence of {\it Lehmer numbers} is defined by 
$$\mathcal{L}_\ell(\alpha, \beta)=\begin{cases}
\dfrac{\alpha^\ell-\beta^\ell}{\alpha-\beta} & \text{ if } \ell \text{ is odd}, \vspace{1mm}\\
\dfrac{\alpha^\ell-\beta^\ell}{\alpha^2-\beta^2} & \text{ if } \ell \text{ is even}.
\end{cases}$$
Note that all Lehmer numbers are non-zero rational integers. 
Recall that two Lehmer pairs $(\alpha_1, \beta_1)$ and $(\alpha_2, \beta_2)$ are equivalent if $\alpha_1/\alpha_2=\beta_1/\beta_2\in \{\pm 1, \pm\sqrt{-1} \}$. For such pairs, $\mathcal{L}_\ell(\alpha_1, \beta_1) = \pm \mathcal{L}_\ell(\alpha_2, \beta_2)$ for any positive integer $\ell$. A prime number $p$ is a {\it primitive
divisor} of $\mathcal{L}_\ell(\alpha, \beta)$ if $p\mid\mathcal{L}_\ell(\alpha, \beta)$ and $p\nmid(\alpha^2-\beta^2)^2
\mathcal{L}_1(\alpha, \beta) \mathcal{L}_2(\alpha, \beta) \cdots \mathcal{L}_{\ell-1}(\alpha, \beta)$. 

We now define $a:=(\alpha+\beta)^2$ and $ b:=a-4\alpha\beta$. Then $\alpha=(\sqrt{a}\pm\sqrt{b})/2$ and $\beta=(\sqrt{a}\mp\sqrt{b})/2$. This pair $(a, b)$ is known as the parameter of the Lehmer pair $(\alpha, \beta)$. We combine a classical result of  Voutier \cite[Theorem 1]{VO95} and a result of Bilu et al. \cite[Theorem 1.3]{BH01}; precisely Tables 2 and 4 in \cite{BH01}, to get the following: 
\begin{thma}\label{thmVO}
Let $\ell$ be a prime number such that $3\leq\ell\leq 30$. If the Lehmer numbers $\mathcal{L}_\ell(\alpha, \beta)$ have no primitive divisor, then up to equivalence, the parameters $(a, b)$ of the corresponding Lehmer pair $(\alpha, \beta)$ are given by the following:
\begin{itemize}
\item[(i)] for $\ell=3, (a, b)=\begin{cases} (1+t, 1-3t) \text{ with } t\ne 1,\\ 
(3^k+t, 3^k-3t) \text{ with } t\not\equiv 0\pmod 3, (k,t)\ne(1,1);\end{cases}$
\item[(ii)] for $\ell=5, (a, b)=\begin{cases} (F_{k-2\varepsilon}, F_{k-2\varepsilon}-4F_k)\text{ with } k\geq 3, \\
 (L_{k-2\varepsilon}, L_{k-2\varepsilon}-4L_k)\text{ with } k\ne 1;\end{cases}$
\item[(iii)] for $\ell=7, (a, b)=(1,-7), (1, -19), (3, -5), (5, -7), (13, -3), (14, -22)$;  
\item[(iv)] for $\ell=13, (a, b)=(1,-7)$;
\end{itemize}
where $t$ is any non-zero integer, $k$ is non-negative integer, $\varepsilon=\pm 1$, and $F_{k}$ (resp. $L_k$) denotes the $k$-th Fibonacci (resp. Lucas) number.  
\end{thma}
We also need the following result of Bilu et al. \cite[Theorem 1.4]{BH01} in the proofs. 
\begin{thma}\label{thmBHV}
For every integer $\ell>30$, the Lehmer numbers $\mathcal{L}_\ell (\alpha, \beta) $ have primitive divisors.
\end{thma}
Let $F_k$ (resp. $L_k$) denote the $k$-th term in the Fibonacci (resp. Lucas) sequence defined by $F_0=0,   F_1= 1$,
and $F_{k+2}=F_k+F_{k+1}$ (resp. $L_0=2,  L_1=1$, and $L_{k+2}=L_k+L_{k+1}$), where $k\geq 0$ is an integer. Then 
$L_k=F_{k-1}+F_{k+1}$ and $5F_k=L_{k-1}+L_{k+1}$ for all $k\geq 1$. Applying these relations, we can derive the following:

\begin{lem}\label{flp}
For an integer $k\geq 0$, let $F_k$ (resp. $L_k$) denote the $k$-th Fibonacci (resp. Lucas) number. Then for $\varepsilon=\pm 1$,
\begin{itemize}
\item[(i)] $4F_k-F_{k-2\varepsilon}=L_{k+\varepsilon}$;
\item[(ii)] $4L_k-L_{k-2\varepsilon}=5F_{k+\varepsilon}$.
\end{itemize}
\end{lem}
\begin{thma}[{\cite[Theorems 2 and 4]{CO64}}]\label{thmCO} Let $F_k$ and $L_k$ be as in Lemma \ref{flp}. Then
\begin{itemize}
\item[(i)] if $L_k=2x^2$, then $k=0,6$;
\item[(ii)] if $F_k=2x^2$, then $k=0, 3,6$.
\end{itemize}
\end{thma}

The following classical lemma is a special case of \cite[Corollary 3.1]{YU05}.
\begin{lem}\label{leyu}
Let $c$ and $d$ be positive integers  such that $cd$ is square-free and $cd\not\equiv 3\pmod 4$. Assume that $n$ is a positive odd integer such that $\gcd(n, h(-cd))=1$. Then all the positive integer solutions $(X,Y,Z)$ of the equation 
\begin{equation}\label{eqyu}
cX^2+dY^2=2Z^n,~~ \gcd(cX, dY)=1,
\end{equation} 
can be expressed as 
$$\frac{X\sqrt{c}+Y\sqrt{-d}}{\sqrt{2}}=\delta\left(\frac{u\sqrt{c}\pm v\sqrt{-d}}{\sqrt{2}}\right)^n,$$
where $u$ and $v$ are positive integers satisfying $Z=\dfrac{u^2c+v^2d}{2}$ and $\gcd(uc, vd)=1$, and $\delta\in\{-1, 1\}$.
\end{lem}

\section{Proofs}
\begin{proof}[\bf Proof of Theorem \ref{thm}] Suppose that $( x, y, m, p)$ is a positive integer solution of \eqref{eqp} for the given  integers $c$ and $d$ as in Theorem \ref{thm}. We first observe that $2\nmid cdx$ since $cd$ is square-free. Also reading \eqref{eqp} modulo $4$ and then applying $cd\not\equiv 3\pmod 4$, we see that $y$ is odd.   

Since $cd$ is square-free, so that $\gcd(cx,d)=1$, and thus by Lemma \ref{leyu}, we can write:
\begin{equation}\label{eqp1}
\frac{x\sqrt{c}+d^m\sqrt{-d}}{\sqrt{2}}=\delta\left(\frac{u\sqrt{c}\pm v\sqrt{-d}}{\sqrt{2}}\right)^p,
\end{equation} 
where $u$ and $v$ are positive integers satisfying 
\begin{equation}\label{eqp2}
y=\dfrac{u^2c+v^2d}{2}
\end{equation} and $\gcd(uc, vd)=1$, and $\delta\in\{-1, 1\}$. Note that $2\nmid uv$ since $2\nmid cdxy$.
 
Assume that $\alpha=\dfrac{u\sqrt{c}+v\sqrt{-d}}{\delta\sqrt{2}}$ and $\bar{\alpha}=\dfrac{u\sqrt{c}- v\sqrt{-d}}{\delta\sqrt{2}}$. Then using \eqref{eqp2}, we check that both $\alpha$ and $\bar{\alpha}$ are algebraic integers. Also, $(\alpha+\bar{\alpha})^2=2u^2c$ and $\alpha\bar{\alpha}=\dfrac{u^2c+v^d}{2}=y$ are coprime positive integers since $\gcd(2u^2c, y)=1$. We see that $\alpha/\bar{\alpha}$ is a root of $$yZ^2-(u^2c-v^2d)Z+y=0.$$ As $\gcd(y, u^2c-v^2d)=\gcd(uc, vd)=1$, so that $\alpha/\bar{\alpha}$ is not an algebraic integer and thus it not a root of unity. Therefore $(\alpha, \bar{\alpha})$ is a Lehmer pair with the corresponding parameters $(2u^2c, -2v^2d)$. 

Suppose that $\mathcal{L}_\ell(\alpha, \bar{\alpha}), \ell\in \mathbb{N}$, is the Lehmer number corresponding to the Lehmer pair $(\alpha, \bar{\alpha})$. Then using \eqref{eqp1}, we get

\begin{equation}\label{eqp3}
\mathcal{L}_p(\alpha, \bar{\alpha})=\frac{\alpha^p-\bar{\alpha}^p}{\alpha-\bar{\alpha}}=\frac{\delta d^m}{v}.
\end{equation}  
This shows that all the prime divisors of $\mathcal{L}_p(\alpha, \bar{\alpha})$ are also the divisors of $d$. On the other hand $(\alpha^2-\bar{\alpha}^2)^2=-2u^2v^2cd$, which shows that all the divisors of $d$ are also divisors of $(\alpha^2-\bar{\alpha}^2)^2$. Therefore $\mathcal{L}_p(\alpha, \bar{\alpha})$ has no primitive divisor. 

Since $(2u^2c, -2v^2d)$ is the parameter corresponding to the Lehmer pair $(\alpha, \bar{\alpha})$, so that by Theorems \ref{thmVO} and \ref{thmBHV} there is no Lehmer number without primitive divisor for $p>7$. This shows that \eqref{eqp} has no positive integer solution when $p>7$. 

We now treat the remaining cases , viz. $p=3,5,7$, individually. For $p=7$, by Theorem \ref{thmVO} the only possibility is $(2u^2c,-2v^2d)=(14, -22)$, which gives $(c, d, u, v)=(7,11,1,1)$.  Using this in \eqref{eqp2}, we get $y=9$. Thus \eqref{eqp} reduces to 
$7x^2+11^{2m+1}=2\times 9^7$. 
We use MAGMA to solve it, which gives $(x,m)=(1169, 0)$. Therefore the corresponding solution of \eqref{eqp} is 
$(c,d,x,y,m, p)=(7,11,1169,9, 0,7)$. 

We now consider the case $p=5$. In this case, Theorem \ref{thmVO} gives
\begin{equation}\label{eqp4}
\begin{cases}
F_{k-2\varepsilon}=2u^2c,\\
F_{k-2\varepsilon}-4F_k=-2v^2d,
\end{cases}
\end{equation}
and 
\begin{equation}\label{eqp5}
\begin{cases}
L_{t-2\varepsilon}=2u^2c,\\
L_{t-2\varepsilon}-4L_t=-2v^2d,
\end{cases}
\end{equation}
where $k\geq 3$ and $t\geq 2$ are integers, and $\varepsilon=\pm 1$. Now \eqref{eqp4} gives $2F_k=u^2c+v^2d$, and thus by \eqref{eqp2}, we get $F_k=y$ with $k\geq 3$. Also  \eqref{eqp4} and Lemma \ref{flp} together give $c=F_{k-2\varepsilon}/2u^2$ and $d=L_{k+\varepsilon}/2v^2$. 

We now compare the real parts in \eqref{eqp1} for $p=5$ to get 
$$x=\frac{\delta u}{4}(u^4c^2-10u^2v^2cd+5v^4d^2).$$
Applying \eqref{eqp4} and then simplifying, we have
$$x=
\delta u\left((2F_k-F_{k-2\varepsilon})^2+4F_k^2-F_kF_{k-2\varepsilon}\right).$$
We apply Lemma \ref{flp} (precisely (i)) to get 
$$x=
u\left((2F_k-F_{k-2\varepsilon})^2+F_kL_{k+\varepsilon}\right).$$

As in the previous case, \eqref{eqp1}, \eqref{eqp2}, \eqref{eqp5} and (ii) of Lemma \ref{flp} all together give
$$(c,d,x,y)=\left(L_{t-2\varepsilon}/2u^2, 5F_{t+\varepsilon}/2v^2, u\left( (2L_t-L_{t-2\varepsilon})^2+5L_tF_{t+\varepsilon}\right), L_t\right),$$
where $t\geq 2$. 

We now consider the remaining case, $p=3$. In this case, we get by Theorem \ref{thmVO}: 
\begin{equation}\label{eqp6}
\begin{cases}
2u^2c=a+1,\\
2v^2d=3a-1,
\end{cases}
\end{equation}
and 
\begin{equation}\label{eqp7}
\begin{cases}
2u^2c=b+3^k,\\
2v^2d=3b-3^k,
\end{cases}
\end{equation}
where $a>1, b>1$ and $k>1$ are integers with $b\not\equiv 0\pmod 3$. The equations 
\eqref{eqp2} and \eqref{eqp6} together give $d=(3u^2c-2)/v^2$ and $y=2u^2c-1$. Also comparing the real and imaginary parts of \eqref{eqp1}, and then applying \eqref{eqp6}, we get $x=u(4u^2c-3)$ and $2d^m=\delta (3u^2cv-dv^3)$.
Since $v$ is a divisor of $d^m$, so that $v=d^r$ for $0\leq r\leq m$. Thus we can write the last equation as 
\begin{equation}\label{eqp8}
2d^{m-r}=\delta(3u^2c-d^{2r+1}).
\end{equation} 
This shows that $d=1, 3$, except for $r=m$.  
If $d=1$, then \eqref{eqp6} implies $3a-1=2$, which gives $a=1$ and thus $c=1$. Therefore, we have $(c,d,x,y)=(1,1,1,1)$.

Again for $d=3$, \eqref{eqp8} can be written as $2\times 3^{m-r-1}=\delta(u^2c-3^{2r})$, which implies either $r=0$ or $r=m-1$. If $r=0$, then $v=1$ and thus as before, we get $(c,d,x,y)=(1,1,1,1)$. For $r=m-1$, \eqref{eqp6} implies $2\times 3^{2m-1}=3a-1$, which is not possible. Finally for $r=m$,  we have $$(c,d,x,y)=\left(c, (3u^2c-2)/v^2, u(4u^2c-3),2u^2c-1\right),$$
where $v=d^m$ and $u\geq 1$ is any odd integer.

We now compare the imaginary parts in \eqref{eqp1}, and then use \eqref{eqp7} to get
$d=3$ and $m=k>1$. Thus using \eqref{eqp2} and \eqref{eqp7}, we get
$y=2v^2+3^{k-1}$ and $c=(v^2+2\times 3^{k-1})/u^2$. Finally, equating the real parts in \eqref{eqp1} and using these values, we have
$x=u(4v^2-2\times 3^{k-1})$. This is not possible since $x$ is odd. 
\end{proof}

\begin{proof}[\bf Proof of Theorem \ref{thm2}]
Let $p$ be a prime divisor of $n$. Then \eqref{eqn} can be written as 
\begin{equation}\label{eqn1}
\ell x^2+q^{2m+1}=2Y^p,
\end{equation}
where $Y=y^{n/p}$. Thus by Theorem \ref{thm}, \eqref{eqn1} has no solution when $p>7$, and therefore \eqref{eqn} has no solution when $n$ has a prime divisor $>7$.

For $p=7$, Theorem \ref{thm} gives $(\ell, q, x, y, m)=(7,11,1169,9, 0)$. Since $Y=y^{n/7}=9$, so that $(y,n)=(9, 7), (3, 14)$. 
 
We now consider the case $p=5$. As in \eqref{eqp1}), we get
\begin{equation}\label{eqn2}
\frac{x\sqrt{\ell}+q^m\sqrt{-q}}{\sqrt{2}}=\delta\left(\frac{u\sqrt{\ell}\pm v\sqrt{-q}}{\sqrt{2}}\right)^5,
\end{equation} 
where $u$ and $v$ are positive integers satisfying 
$Y=y^{n/5}=\dfrac{u^2\ell+v^2q}{2}$
with $\gcd(uc, vd)=1$, and $\delta\in\{-1, 1\}$. Also analogous to \eqref{eqp3}, one gets $\mathcal{L}_5(\alpha, \bar{}\alpha)=q^m/v$, which implies $v=q^r$ for some integer $0\leq r \leq m$. We first assume that $r\leq m-1$. Then the imaginary parts of \eqref{eqn2} gives $q=5$ since $\gcd(q, u\ell)=1$. Upon simplifying these imaginary parts, we get
$$u^4\ell^2-5^{2r}\times 2u^2\ell +5^{4r+1}=5^{m-r-1}\times 4\delta,$$  
which implies that $r=m-1$, and thus $v=5^{m-1}$. Therefore the above equation can be written as
$$u^4\ell^2-5^{2m-2}\times 2u^2\ell+5^{4m-3}=4\delta.$$
This can be simplified to 
$$\left(\frac{u^2\ell-5^{2m-2}}{2}\right)^2+5^{4m-4}=\pm 1.$$
Reading this modulo $4$, we can further reduce it to  $$\left(\frac{u^2\ell-5^{2m-2}}{2}\right)^2+5^{4m-4}=1, \text{ with }\ell\equiv 1\pmod 4.$$
This can be written as 
$$\left(\frac{2+5^{2m-2}-u^2\ell}{2}\right)\left(\frac{2+u^2\ell-5^{2m-2}}{2}\right)=5^{4m-4}.$$
Clearly, $(2+5^{2m-2}-u^2\ell)/2$ and $(2+u^2\ell-5^{2m-2})/2$ are coprime, so that either $(2+5^{2m-2}-u^2\ell)/2=1$ or $(2+u^2\ell-5^{2m-2})/2=1$. Both of them lead to $u^2\ell=5^{2m-2}$, which is not possible as $\gcd(u\ell, 5)=1$.

When $r=m$, we have  $|\mathcal{L}_5(\alpha, \bar{}\alpha)|=1$ and thus it has no primitive divisor. Therefore by Theorem \ref{thmVO} and Lemma \ref{flp}, we get 
\begin{equation}\label{5x}
(F_{k-2\varepsilon}, L_{k+\varepsilon})=(2u^2\ell, 2q^{2m+1}),
\end{equation}
or
\begin{equation}\label{5xx}
(L_{k-2\varepsilon}, 5F_{k+\varepsilon})=(2u^2\ell, 2q^{2m+1}).
\end{equation}
Now \eqref{5xx} gives $q=5$, and hence $F_{k+\varepsilon}=2\times 5^{2m}$. Thus by Theorem \ref{thmCO}, $(k,\varepsilon, m)=(4,-1,0),$ $ (2,1,0)$, and hence using \eqref{5xx}, we get $(u, \ell)=(1,3)$. Therefore $Y=(u^2\ell+v^2q)/2$ implies that $Y=4$, and thus \eqref{eqn1} becomes 
$$3x^2+5=2\times 4^5.$$
This equation has no integer solution. 

Now the imaginary parts of \eqref{eqn2} gives
$$5u^4\ell^2-10u^2\ell q^{2m+1}+q^{4m+2}=\pm 4.$$
This can be written as 
$$\left(\frac{q^{2m+1}-5u^2\ell}{2}\right)^2-5u^4\ell^2=\pm 1.$$
By reading this equation modulo $4$, it can be reduced to 
$$\left(\frac{q^{2m+1}-5u^2\ell}{2}\right)^2-5u^4\ell^2=-1.$$
Employing \eqref{5x}, we get 
\begin{equation}\label{5xxx}
\left(\frac{L_{k+\varepsilon}-5F_{k-2\varepsilon}}{4}\right)^2-\frac{5}{4}F_{k-2\varepsilon}^2=-1.
\end{equation}
 Again equating the real parts of \eqref{eqn2}, we get
 $$4x=|u(u^4\ell^2-10u^2\ell v^2q+5v^4q^2)|.$$
 Applying \eqref{5x}, this can be written as 
 \begin{equation*}\label{5xxxx}
 x=\frac{1}{4}\left|\sqrt{\frac{F_{k-2\varepsilon}}{2\ell}} \left(F_{k-2\varepsilon}^2-\frac{5}{2}F_{k-2\varepsilon}L_{k+\varepsilon}+\frac{5}{4}L_{k+\varepsilon}^2\right)\right|.
 \end{equation*} 
Now $y^{n/5}=(u^2\ell+v^2q)/2$ gives
\begin{equation*}\label{5xxxxx}
y=\left(\frac{F_{k-2\varepsilon}+L_{k+\varepsilon}}{4}\right)^{5/n}.
\end{equation*} 

Finally for $p=3$, we follow the previous approach to get 
$$(\ell, q,x,Y)=\left(\frac{q^{2m+1}+2}{3u^2}, q, \frac{u(4q^{2m+1}-1)}{3}, \frac{2q^{2m+1}+1}{3}\right).$$
This implies 
$$(\ell, q,x,y)=\left(\frac{q^{2m+1}+2}{3u^2}, q, \frac{u(4q^{2m+1}-1)}{3}, \left(\frac{2q^{2m+1}+1}{3}\right)^{1/2^r3^t}\right),$$
where $r\geq 0$  and $t\geq 0$ are suitable integers.   
%
%
\end{proof}

\begin{proof}[Proof of Theorem \ref{thm3}] The proof of this theorem is {\it mutatis mutandis} the same as that of Theorem \ref{thm2} when $n$ has a prime divisor $>5$. For the remaining cases, we rewrite \eqref{eqd} as 
\begin{equation}\label{eqd1}
x^2+d^{2m+1}=2Y^p,
\end{equation}
where $p\leq 5$ is a prime divisor of $n$ and $Y=y^{n/p}$. Note that both $x$ and $y$ are odd as $d\equiv 1\pmod 4$.  

Assume that $p=5$. Then by Theorem \ref{thm}, $F_{k-2\varepsilon}=2u^2$ and $L_{{t-2\varepsilon}}=2u^2$. Since $u\geq 1$ is odd, $k\geq 3$ and $t\geq 2$, so that by Theorem \ref{thmCO} we get: 
$$(k,\varepsilon, u)=(5,1,1)\text{ and } (t,\varepsilon, u)\in\left\{(2,1,1), (4,-1,3), (8,1,3)\right\}.$$

Also by Theorem \ref{thm}, we get $L_{k+\varepsilon}=2v^2d$ and $5F_{t+\varepsilon}=2v^2d$, which further imply:
$$(k,\varepsilon,u, v,d)=(5,1,1,3,1)\text{ and }(t,\varepsilon, u, v, d)\in\left\{(2,1,1,1,5), (4,-1,3,1,5), (8,1,3,1,85)\right\}.$$
Utilizing these values in 
$$x=\begin{cases}u\left((2F_k-F_{k-2\varepsilon})^2+F_kL_{k+\varepsilon}\right),\\
u\left( (2L_t-L_{t-2\varepsilon})^2+5L_tF_{t+\varepsilon}\right),
\end{cases}$$
we see that $x$ is even in each case. This is not possible, and thus \eqref{eqd1} has no solution when $p=5$. 

We now assume that $p=3$. In this case, we get by Theorem \ref{thm}: 
$$(x,Y,d, m)=(4u^3-3u, 2u^2-1, (3u^2-2)/v^2, m),$$
where $v=d^m$ and $u\geq $ is an odd integer.  
Thus the corresponding solutions of \eqref{eqd} are given by $(x,y,d,m)=\left(4u^3-3u, (2u^2-1)^{3/n}, (3u^2-2)/v^2, m\right)$, where $n=3t$ with $t$ a positive integer having only $1,2$ or $3$ its possible divisors.
\end{proof}
\section*{acknowledgements}
The author would like to thank Professor  Srinivas Kotyada and The Institute of Mathematical Sciences, Chennai for hospitality and support during the period when this work was started. 
The author thanks Professor Kalyan Chakraborty and Professor Alain Togb\'e for comments and suggestions which have improved the quality of the paper. The author is grateful to the anonymous referee for valuable suggestions and particularly for drawing the paper \cite{GM20} to his attention.
This work is supported by the grants SERB MATRICS Project No. MTR/2021/000762, Govt. of India.


\begin{thebibliography}{25}
\bibitem{MLST09} F. S. Abu Muriefah, F. Luca, S. Siksek and Sz. Tengely, {\it On the Diophantine equation $x^2 + C = 2y^n$},  
Int. J. Number Theory {\bf 5} (2009), no. 6, 1117--1128.


\bibitem{BH01} Y. Bilu, G. Hanrot and F. M. Voutier, {\it Existence of primitive divisors of Lucas and Lehmer numbers (with an appendix by M. Mignotte)}, J. Reine Angew. Math. {\bf 539} (2001), 75--122.

\bibitem{CO64}  J. H. E. Cohn, {\it Square Fibonacci numbers, etc.}, Fibonacci Quart. {\bf 2} 1964, no. 2, 109--113.

\bibitem{CO96} J. H. E. Cohn, {\it Perfect Pell powers}, Glasgow Math. J. {\bf 38} (1996),19--20.
\bibitem{GM20} N. Ghanmi and F. S. Abu Muriefah, {\it On the Diophantine equation $Cx^2+D=2y^q$}, Ramanujan J. {\bf 53} (2020), 389--397.

\bibitem{LJ66} W. Ljunggren, {\it  On the Diophantine equation $Cx^2 + D = 2y^n$},  Math. Scand. {\bf 18} (1966), 69--86.

\bibitem{PT00} I. Pink and Sz. Tengely, {\it Full powers in arithmetic progressions}, Publ. Math. Debrecen {\bf 57} (2000, 535--545.

\bibitem{ST95} C. St\"ormer,  {\it Solution compl\`ete en nombres entiers $m, n, x, y, k$  de l'\'equation $m \arctan (x^{-1})+n\arctan( y^{-1})=k\pi/4$}, Norske Videnskabsselsk. Skr. Kristiania {\bf 1} (1895), no. 11, 21 pp.

\bibitem{TE04} Sz. Tengely, {\it On the Diophantine equation $x^2+ a^2 = 2y^p$},  Indag. Math. (N. S.) {\bf 15} (2004), no. 2,  291--304.

\bibitem{TE07} Sz. Tengely, {\it On the Diophantine equation $x^2 + q^{2m} = 2y^p$}, Acta Arith. {\bf 127} (2007), no. 1, 71--86.

\bibitem{VO95} P. M. Voutier, {\it Primitive divisors of Lucas and Lehmer sequences}, Math. Comp. {\bf 64} (1995),
869--888.


\bibitem{YU05} P. -Z. Yuan, {\it On the Diophantine equation $ax^2 + by^2 = ck^n$}, Indag. Math. (N. S.) {\bf 16} (2005), no. 2,
301--320.

\bibitem{ZLT12} H. Zhu, M. -H. Le and A. Togb\'{e},  {\it On the exponential Diophantine equation $x^2 + p^{2m} = 2y^n$}, Bull. Aust. Math. Soc. {\bf 86} (2012), 303--314.
\end{thebibliography}
\end{document}